\documentclass[11pt,a4paper,reqno]{amsart}
\usepackage{amsmath}
\usepackage{amssymb,latexsym}
\usepackage[dvips]{graphicx}
\usepackage{color}
\usepackage{cite}
\usepackage{amsthm}

\newcommand{\R}{\mathbb R}

\newtheorem{theorem}{Theorem} [section]
\newtheorem{lemma}{Lemma} [section]
\newtheorem{proposition}{Proposition} [section]
\newtheorem{corollary}{Corollary} [section]

\newtheorem{remark}{Remark}[section]

\let\ssection=\section\renewcommand{\section}{\setcounter{equation}{0}\ssection}
\begin{document}
\date{}
\address{Mohamad Darwich: Facult\'e des sciences de l'Universit\'e Libanaise, D\'epartement de math\'ematiques pures, Hadath, Liban} \email{Mohamad.Darwich@lmpt.univ-tours.fr}
\address{Luc Molinet: Laboratoire de Math\'ematiques et Physique Th\'eorique\\
Universt\'e Francois Rabelais Tours,
CNRS UMR 7350- F\'ed\'eration Denis Poisson\\
Parc Grandmont, 37200 Tours, France}
\email{Luc.Molinet@lmpt.univ-tours.fr}
\title[Blowup]{Some remarks on the nonlinear Schr\"odinger Equation with  fractional dissipation.}
\author{Mohamad Darwich and Luc Molinet. }

\begin{abstract}
We consider the Cauchy problem for the  $L^{2}$-critical  focussing nonlinear Schr\"{o}dinger equation with a fractional dissipation.
 According to the order of the fractional dissipation, we prove  the global existence  or the
existence of finite time blowup dynamics with the log-log blow-up speed  for $\|\nabla u(t)\|_{L^2}$.
\end{abstract}

\maketitle
\section{Introduction}
In this paper, we study  the blowup and the global existence of solutions for the $L^2$-critical nonlinear Schr\"odinger equation (NLS) with a fractional dissipation term:
\begin{equation}
\left\{
\begin{array}{l}
iu_{t} + \Delta{u} +|u|^{\frac{4}{d}}u + ia(-\Delta)^s u =0,  (t,x) \in [0,\infty[\times \mathbb{R}^{d}, d=1,2,3,4. \\
u(0)= u_{0} \in H^r(\mathbb{R}^{d})
\end{array}
\right. \label{NLSas}
\end{equation}
  where $a > 0$ is the coefficient of friction,  $ s\geq 0 $ and  $ r\in \R $.
  
The NLS equation ($a=0$) arises in various areas of nonlinear optics, plasma physics and fluid mechanics to describe propagation phenomena in dispersive media.
To take into account weak dissipation effects, one usually add a damping term as in the  linear damped NLS equation (see for instance Fibich \cite{FibichG} ):
$$i u_t +\Delta u +ia  u + |u|^p  u =0 , a>0,$$
or a laplacian term as in the following complex Ginzburg-Landau equation studied in Passota-Sulem-Sulem \cite{passot}:.

$$i u_t +\Delta u -ia \Delta  u + |u|^p  u =0   , a>0,$$

However, in many cases of practical importance the damping is not described by a  local term even in the long-wavelenth limit. In media with dispersion the weak dissipation is, in general, non local (see for instance Ott-Sudan \cite{Ott}). It is thus quite natural to complete the NLS equation by a non local dissipative term in order to take into account some dissipation phenomena. 

In this this paper we complete the $L^2$-critical NLS equation (\ref{NLS}) with a  fractional laplacian of order $ 2s $, $ s>0 $,  and study the influence of this term on the blow-up phenomena for this equation.
\\
Recall that  the Cauchy problem for  the $L^2$-critical focussing nonlinear Schr\"{o}dinger equation ($a=0$):\\
\begin{equation}
\left\{
\begin{array}{l}
 iu_{t} + \Delta u + |u|^{\frac{4}{d}}u = 0\\
 u(0)=u_{0} \in H^{r}(\mathbb{R}^{d})
 \end{array}
\right.\label{NLS}
 \end{equation}
has been  studied by a lot of authors (see for instance \cite{Kato}, \cite{CW2}, \cite{Cazenave1}) and it is known that the problem is locally well-posed in $H^r(\mathbb{R}^{d})$ for $ r\ge 0 $ : 
For any $u_{0} \in H^{r}(\mathbb{R}^{d})$, with $ r\ge 0 $, there exist $T >0 $ and a unique solution $u$ of $(\ref{NLS})$
 with $u(0)=u_{0}$ such that $u \in C([[0,T]);H^r(\mathbb{R}^{d}))$. Moreover, if $T^*$  is the maximal existence time of the solution $u$ in $ H^r(\R^d) $ 
 then $\displaystyle{ \lim_{t\rightarrow T^*}{\|u(t)\|_{L^2(\mathbb{R}^{d})}}}=\infty$.\\

Let us mention that in the case $ a > 0$ the same results on the Cauchy problem for (\ref{NLSas}) can be established  in exactly the same way as in the case $a = 0$, since the same Strichartz estimates hold (see for instance  \cite{Zhang}).

  For $u_0 \in  H^1(\R^d)$, a sharp criterion for global existence for (\ref{NLS}) has been exhibited by Weinstein
\cite{weinst}:  Let  $Q$  be the unique radial positive solution to 
 \begin{equation}\label{ellip}
 \Delta Q + Q|Q|^{\frac{4}{d}} = Q.
 \end{equation}
If $\|u_0\|_{L^2} < \|Q\|_{L^2}$ then  the solution of (\ref{NLS}) is global in $H^1$. This follows from the conservation
of the energy and the $L^2$ norm and the sharp Gagliardo-Nirenberg inequality which ensures that 
\begin{equation}\label{ZZ}
\forall u \in H^1,  \; E(u) \geq \frac{1}{2}(\int |\nabla u|^2)\bigg(1 - \big(\frac{\int |u|^2}{\int |Q|^2}\big)^{\frac{2}{d}}\bigg). \
\end{equation} 
Actually, it was recently proven that any solution of \eqref{NLS} emanating from an initial datum $ u_0\in L^2(\R^d) $ with $ \|u_0\|_{L^2} <\|Q\|_{L^2}$ 
is global and does scatter (cf. \cite{dodson1}).

On the other hand, there exists explicit solutions with $\|u_0\|_{L^2} =\|Q\|_{L^2}$ that blow up at some  time $T>0  $  with a $ H^1$norm that grows as $\frac{1}{T-t}$.\\               
In the series of papers \cite{MerleRaphael1,Merle6}, Merle and Raphael  studied the blowup for (\ref{NLS}) with
$ \|Q\|_{L^2}<  \|u_0\|_{L^2} < \|Q\|_{L^2} + \delta$, $\delta$ small
 and proved  the
  existence of the blowup regime corresponding to the log-log law:
\begin{equation}\label{speed}
\displaystyle{\|u(t)\|_{H^1(\mathbb{R}^{d})} \sim \bigg(\frac{\text{log}\left|\text{log}(T-t)\right|}{T-t}\bigg)^{\frac{1}{2}}.}
\end{equation}

Recall that the evolution of (\ref{NLS}) admits the following conservation laws in the
energy space $H^1$:\\
$L^2$-norm : $m(u)=\left\|u\right\|_{L^2} = \Bigl( \displaystyle{\int |u(x) |^2dx}\Bigr)^{1/2}$.\\
Energy : $E(u) = \frac{1}{2}\|\nabla u\|_{L^2}^{2} - \frac{d}{4 + 2d}\|u\|_{L^{\frac{4}{d}+2}}^{\frac{4}{d}+2}. $\\
Kinetic momentum : $P(u)=Im(\displaystyle{\int} \nabla u(x) \overline{u}(x)\, dx .$\\
Now, for  (\ref{NLSas}) with $ a>0$, there does not exist conserved quantities anymore.
However, it is easy to prove that if $u$ is a  smooth solution of (\ref{NLSas}) on $[0,T[ $, then for all $ t\in [0,T[ $ it holds
\begin{equation}\label{masse}
\|u(t)\|_{L^2}^2 + a \int_0^t \|(-\Delta)^{\frac{s}{2}}u\|^2_{L^2} = \|u_0\|_{L^2}^2 \; .
\end{equation}
\begin{equation}\label{energie}
\frac{d}{dt} E(u(t)) =  \displaystyle{-a\int ((-\Delta)^{\frac{s+1}{2}}u(t))^2 +a Im \int (-\Delta)^s u(t) |u(t)|^{\frac{4}{d}}\overline{u}(t)}.
\end{equation}
\begin{equation}\label{momentum}
\displaystyle{\frac{1}{2} \frac{d}{dt} P(u(t)) = a Im \int (-\Delta)^{s}u(t)\overline{\nabla u}(t)}.
\end{equation}
In  \cite{Darwich}, the first author  studied the case $ s=0$. He  proved the global existence  in $H^1$ for
$\|u_0\|_{L^2} \leq \|Q\|_{L^2}$, and  showed that the log-log regime is stable by  such perturbations (i.e. there exist  solutions 
 blows up in finite time with the log-log law).\\
 In  \cite{passot}, Passot, Sulem and  Sulem  proved that 
the solutions  are global in $H^1(\mathbb{R}^{2})$ for $s=1$. However, their method does not seem to apply for any other values of $d$.\\
Our aim in 
this paper  is to establish some results, for  $ s>0 $, on  the global existence or the existence
of finite time blowup dynamics with the log-log blow-up speed for $\|\nabla u \|_{L^2}$.
\vspace{0,3cm} \\
Let us now state our results:

\begin{theorem}\label{theorem1}
Let $d=1,2,3,4$ and  $0 < s < 1$ then there exists $\delta_0 > 0$ such that $\forall a > 0$ and $\forall \delta \in]0, \delta_{0}[$,
 there exists $u_0 \in H^1$ with $\|u_0\|_{L^2} = \|Q\|_{L^2} + \delta$, such that the solution of (\ref{NLSas}) blows up in finite time in the log-log regime.
\end{theorem}
\begin{theorem}\label{firsttheorem}
Let $d=1,2,3,4$, $s \geq 1$ and $ r \ge 0 $.  Then the Cauchy problem (\ref{NLSas}) is globally well-posed in $ H^r(\R^d) $.
\end{theorem}
\begin{theorem}\label{theorem4} Let $ 0<s<1 $ and $ a>0 $. 
\begin{enumerate}
\item There exists a real number $0< \gamma=\gamma(d)\le \|Q\|_{L^2} $ such that for any initial datum $u_0 \in H^1(\R^d) $ with $\|u_0\|_{L^2} <\gamma $, the emanating solution u is global in $ H^1$ with an energy that is non increasing.
\item
 There does not exists any intial datum $u_0$, with $\|u_0\|_{L^2} \leq \|Q\|_{L^2}$, such that the solution $u$ of (\ref{NLSas})
 blows up at finite time $ T^* $  and satisfies 
$$\frac{1}{(T^*-t)^{\alpha }} \lesssim \|\nabla u(t)\|_{L^2(\mathbb{R}^{d})} \lesssim \frac{1}{(T^*-t)^{\beta }},\quad \forall \, 0<T-t\ll 1\, , 
$$
for some pair $(\alpha,\beta) $ satisfying  $0	<\beta <\frac{1}{2s}  $ and $  \beta(1+s) -1/2 <\alpha \le \beta $. 
\end{enumerate}
\end{theorem}
\begin{remark}
 Note that, assertion (2) of Theorem \ref{theorem4} ensures that
 we do not  have any blowup in the  log-log regime for any $0<s<1$ and in the regime $\frac{1}{t}$ for any $0<s<\frac{1}{2}$, for initial data with 
 critical or subcritical mass.
\end{remark}
Acknowledgments : The first author thanks the L.M.P.T. for his kind  hospitality during the development of this work. Moreover, he would like to thank AUF for supporting this project.

\section{Local and global existence results}

In this section, we prove Theorem \ref{firsttheorem} and  part (1) of Theorem \ref{theorem4}. Theorem \ref{firsttheorem} will follow from  an a priori estimate on the critical Strichartz norm whereas  part (1) of Theorem \ref{theorem4} follows from a monotonicity of the energy. 
\subsection{Local existence result}
Recall that the main tools to prove the local existence results for  \eqref{NLS} are the Strichartz estimates for the associated linear propagator $ e^{i\Delta t}� $. 
These Strichartz estimates reads 
$$
\|e^{i\Delta t}�\phi \|_{L^q_t L^r(\R^d)}�\lesssim \|\phi\|_{L^2(\R^d)}
$$
for any  pair $(q,r) $ satisfying   $\frac{2}{q} + \frac{d}{r} = \frac{d}{2}$ and $2 < q \leq \infty $. Such ordered pair is called an admissible pair . \\
For $ a\ge 0 $ and $s\ge 0 $ we denote by $ S_{a,s} $ the linear semi-group associated with \eqref{NLSas}, i.e.
$ S_{a,s}(t)=e^{i\Delta t -a(-\Delta)^s t}$. 
The following lemma ensures that   the linear semi-group $ S_{a,s} $  enjoys the same Strichartz estimates as $ e^{i\Delta t}� $. 
\begin{lemma}\label{LpLp}
Let $\phi \in L^{2}(\mathbb{R}^{d})$. Then  for every admissible pair $(q,r)$ it holds 
$$
\|S_{a,s}(\cdot) \phi \|_{L^q_{t>0} L^r(\R^d)}�\lesssim \|\phi\|_{L^2(\R^d)} \, .
$$
\end{lemma}
\begin{proof}
Setting, for any $ t\ge 0 $,  $\displaystyle{G_{a,s}(t,x) = \int e^{-ix\xi} e^{-at|\xi|^{2s}} d\xi}$, it holds 
$$
S_{a,s}(t)\varphi =G_{a,s}(t,\cdot) \ast e^{it\Delta}\varphi , \quad \forall t\ge 0 \, .
$$
Noticing that  $\|G_{a,s}(t,.)\|_{L^1} = \|G_{1,s}(1,.)\|_{L^1}$ and that, according to 
  Lemma 2.1 in \cite{Zhang}, $G_{1,s}(1,.) \in L^{1}(\mathbb{R}^{d})$, we get 
$$
\|S_{a,s}(t) \phi)\|_{L^{p}_{t>0} L^q(\mathbb{R}^{d})} = \|G_{a,s}(t,.)\ast e^{it\Delta}�(\phi)\|_{L^{p}_{t>0}L^{q}(\mathbb{R}^{d})} \lesssim \|e^{it\Delta}�\phi\|_{L_t^{p}L^q(\mathbb{R}^{d})} 
\lesssim \|\phi\|_{L^2} \, .
$$
\end{proof}
With Lemma \ref{LpLp} in hand, it is not too hard to check that  the local existence results for equation \eqref{NLS} (see for instance \cite{CW2} and \cite{Cazenave1}) also holds for  \eqref{NLSas} with $ a\ge 0 $ and $ s\ge 0 $. More precisely, we have the following statement: 
\begin{proposition}\label{prop1}
Let $ a\ge 0$, $ s\ge 0 $ and $ u_0 \in H^s(\R^d) $ with $ d\in\{1,2,3,4\}$. There exists $ T>0 $ and a unique solution $ u\in C([0,T]; H^s)\cap L^{\frac{4}{d}+2}_T  L^{\frac{4}{d}+2} $ to \eqref{NLSas} emanating from $ u_0 $. In addition, there exists a neighborhood $ {\mathcal V}_{u_0} $ of $ u_0 $ in $ H^s $ such that the associated solution map is continuous from   $ {\mathcal V}_{u_0} $ into $C([0,T]; H^s)\cap L^{\frac{4}{d}+2}_T  L^{\frac{4}{d}+2} $. \\
Finally, let $ T^* $ be the maximal time of existence of the solution $ u $ in $ H^s(\R^d) $, then 
\begin{equation} \label{oo}
T^*<\infty \Longrightarrow \|u\|_{L^{\frac{4}{d}+2}_T  L^{\frac{4}{d}+2} } =+\infty \; .
\end{equation}
 \end{proposition}
\subsection{Proof of Theorem \ref{firsttheorem}}
Let $ u\in C([0,T]; L^2(\R^d) $ be the solution emanating from some initial datum $ u_0\in L^2(\R^d)  $. We have the following a priori estimates:
\begin{lemma}\label{prop1}
Let $ u \in C([0,T];L^2(\R^d)) $ be the  solution of (\ref{NLSas}) emanating from $ u_0\in L^2(\R^d) $. Then
\begin{equation}
\label{ff}
\|u\|_{L^{\infty}_T L^{2}} \leq \|u_0\|_{L^2}�~~ and ~~\|(-\Delta)^{\frac{s}{2}} u\|_{{L^2_T}{L^2}} \leq \frac{1}{\sqrt{2a}} \|u_0\|_{L^2} .
\end{equation}
\end{lemma}
\begin{proof}Assume first that $ u_0 \in H^{\infty} (\R^d) $.  
Then (\ref{masse}) ensures that  the mass is decreasing  as soon as $u$ is not the null solution and  (\ref{masse}) leads to 
$$
\displaystyle{\int_0^T\|(-\Delta)^{\frac{s}{2}}u(t)\|^2_{L^2}\, dt  = -\frac{1}{2a}(\|u(T)\|_{L^2}^2 - \|u_0\|_{L^2}^2)}
\le   \frac{1}{2a }\|u_0\|_{L^2}^2.$$
This proves \eqref{ff} �for smooth solutions. The result for $ u_0 \in L^2(\R^d) $ follows by approximating $ u_0 $ in $ L^2 $ by a smooth sequence $ (u_{0,n})\subset H^\infty(\R^d) $.
\end{proof}
Note that  the first estimate in \eqref{ff} implies that $ \|u\|_{L^2_T L^{2}} \leq T^{1/2} \|u_0\|_{L^2}�$ and thus by interpolation:
\begin{equation}\label{est}
\|\nabla u\|_{L^2_T L^2}\lesssim \|(-\Delta)^{\frac{s}{2}}u\|^{\frac{1}{s}}_{L^2_T L^2}\|u\|^{1-\frac{1}{s}}_{L^{2_T }L^{2}} \lesssim T^{\frac{1}{2}(1-\frac{1}{s})}
\end{equation}
Interpolating now between  (\ref{est}) and the first estimate of \eqref{ff} we get 
$$\|u\|_{L^{\frac{4}{d}+2}_T H^{\frac{2d}{4+2d}}} \lesssim T^{\frac{1}{2}(1-\frac{1}{s})}$$
and the embedding  $H^{{\frac{2d}{4+2d}}}(\R^d) \hookrightarrow L^{\frac{4}{d}+2}(\R^d)$ ensures that 
$$ \|u\|_{L^{\frac{4}{d}+2}_TL^{\frac{4}{d}+2}} \lesssim \|u\|_{L^{\frac{4}{d}+2}_T H^{\frac{2d}{4+2d}}} \lesssim T^{\frac{1}{2}(1-\frac{1}{s})} \, .$$
Denoting by $ T^* $ the  maximal time of existence of $ u $ in $ L^2(\R^d) $ and letting $ T $ tends to $ T^* $, this contradicts \eqref{oo} whenever $ T^* $ is finite.
 This proves that the solutions are global in $ H^r(\R^d) $. 
\subsection{Proof  of Assertion 1 of  Theorem \ref{theorem4}}
Note that  the global existence for any $ u_0\in L^2(\R^d) $  with $ \|u_0\|_{L^2} $ small enough can be proven, as for the critical NLS equation,  directly by a fixed point argument thanks to Lemma \ref{LpLp}.  This ensures the global existence in $ H^r(\R^d) $, $r\ge 0 $, under the same smallness condition 
 on $ \|u_0\|_{L^2} $. We will not invoke this fact here  and we will directly prove 
Assertion 1 of Theorem \ref{theorem4} 	by combining \eqref{ZZ} and a  monotony result on  $t\mapsto  E(u(t)) $. 
To do this, we will  work with smooth solutions and then get the result for $ H^1$-solutions by continuity with respect initial data. 

So, let  $ u\in C([0,T]; H^\infty(\R^d)) $ be  a solution to \eqref{NLSas} emanating from $u_0\in H^\infty(\R^d) $. 
Then it holds 
$$\frac{d}{dt}E(u(t) =  \displaystyle{-a\int |(-\Delta)^{\frac{s+1}{2}}u(t)|^2 +a Im \int (-\Delta)^s u(t)\,  |u(t)|^{\frac{4}{d}}\overline{u(t)}}$$
and H\"older inequalities in physical space and in Fourier space lead to
$$
\Bigl| \int (-\Delta)^s u \, |u|^{\frac{4}{d}}\overline{u} \Bigr| \leq \|(-\Delta)^s u \|_{L^2}\|u\|_{L^{\frac{8}{d}}+2}^{\frac{4}{d} + 1}
$$
with
$$
\|(-\Delta)^s u \|_{L^2} \leq \|(-\Delta)^{\frac{s+1}{2}}u\|_{L^2}^{\frac{2s}{s+1}}\|u\|_{L^2}^{\frac{1-s}{1+s}}\, .
$$
Let us recall the following   Gagliardo-Nirenberg inequality \footnote{It is proven in \cite{weinst} that the constant $C_d $  is related for $d=1,2,3 $ to the $L^2$-norm of the ground state solution of 
   $ 2 \Delta \psi -(\frac{4}{d}-1) \psi + \psi^{\frac{8}{d}+1}=0$. }
$$
\|u\|_{L^{\frac{8}{d}}+2}^{\frac{8}{d} + 2} \leq C_d^{\frac{8}{d}+2} \|\nabla u\|_{L^2}^{4}\|u\|^{\frac{8}{d}-2}_{L^2}\, .
$$
 This estimate together with Cauchy-Schwarz inequality (in Fourier space)
 $$
\|\nabla u\|^{2}_{L^2} \leq \|(-\Delta)^{\frac{s+1}{2}} u \|_{L^2}^{\frac{2}{s+1}}\|u\|_{L^2}^{\frac{2s}{s+1}}
$$
lead to 
$$
\|u\|_{L^{\frac{8}{d}}+2}^{\frac{4}{d} + 1} \leq C^{\frac{4}{d} + 1}\|(-\Delta)^{\frac{s+1}{2}} u \|_{L^2}^{\frac{2}{s+1}}\|u\|_{L^2}^{\frac{2s}{s+1}}\| u\|_{L^2}^{\frac{4-d}{d}}\, .
$$
Combining the above estimates we eventually obtain
$$
\frac{d}{dt}E(u(t)) \leq a \|(-\Delta)^{\frac{s+1}{2}}u\|_{L^2}^2(C^{\frac{4}{d} + 1}\|u\|_{L^2}^{4/d} - 1)
$$
which together with \eqref{ff} implies that $ E(u(t)) $ is not increasing for $ \|u_0\|_{L^2} \le C^{1+\frac{d}{4}} $.

\section{Proof of Assertion 2  of Theorem \ref{theorem4}}
Special solutions play a fundamental role for the description of the dynamics of (NLS). They are the solitary waves of the form $u(t, x) =\exp(it)Q(x)$, where $Q$ the unique positive radial solution to 
 \begin{equation}\label{ellip}
 \Delta Q + Q|Q|^{\frac{4}{d}} = Q.
 \end{equation}
 The
pseudo-conformal transformation applied to the ``stationary'' solution $e^{it}Q(x)$
yields an explicit solution for (NLS)
$$
S(t,x) = \frac{1}{\mid t \mid^{\frac{d}{2}}} Q(\frac{x}{t})e^{-i \frac{\mid x \mid^2}{4t} + \frac{i}{t}}
$$
which blows up at $T^* = 0$.\\
 Note that 
\begin{equation}\label{nablaS}
\|S(t)\|_{L^2} = \|Q\|_{L^2} ~~\text{and}~~ \|\nabla S(t)\|_{L^2}\sim \frac{1}{t}
\end{equation}
 It turns out that
$S(t)$ is the unique minimal mass blow-up solution in $H^1$  up to the symmetries of the equation ( see \cite{Merleseul}).

A known lower bound ( see \cite{Merle4}) on the blow-up rate for (NLS) is 
\begin{equation}\label{lowerNLSap}
\|\nabla u(t)\|_{L^2} \geq \frac{C(u_0)}{\sqrt{T-t}}.
\end{equation}
Note that this blow-up rate is strictly lower than the one of $S(t)$ given by (\ref{nablaS}) and of the log-log law given by (\ref{speed}).\\
To prove assertion 2 of Theorem \ref{theorem4}, we will need the following result ( see \cite{Hmidi}) :\\
\begin{theorem}\label{limsup}
 Let $(v_{n})_{n}$ be a bounded family of $H^1(\mathbb{R}^{d})$, such that:
 \begin{equation}
 \limsup_{n \rightarrow +\infty}\left\|\nabla v_{n}\right\|_{L^2(\mathbb{R}^{d})} \leq M \quad and \quad \limsup_{n \rightarrow +\infty}\left\|v_{n}\right\|_{L^{\frac{4}{d} + 2}} \geq m.
 \end{equation}
 Then, there exists $(x_{n})_{n} \subset \mathbb{R}^{d}$ such that:
 \begin{equation}
 v_{n}(\cdot + x_{n}) \rightharpoonup V \quad weakly, \nonumber
 \end{equation}
 with $\left\|V\right\|_{L^2(\mathbb{R}^{d})} \geq (\frac{d}{d+4})^{\frac{d}{4}}\frac{m^{\frac{d}{2}+1} + 1}{M^{\frac{d}{2}}}\left\|Q\right\|_{L^2(\mathbb{R}^{d})}$.
 \end{theorem}

Suppose that there exist an initial data $u_0$ with $\|u_0\|_{L^2} \leq \|Q\|_{L^2}$ , such that the corresponding solution $u(t)$ blows up
   at time $ T>0 $ with the following behavior: 
    \begin{equation}\label{law}
    \frac{1}{(T-t)^{\alpha }} \lesssim \|\nabla u(t)\|_{L^2(\mathbb{R}^{d})} \lesssim \frac{1}{(T-t)^{\beta}}, \quad \forall t\in [0,T[ ,
    \end{equation}
where $ \beta >0 $  and  $ \alpha \ge \beta $ satisfies $\alpha > \beta(1+s)-1/2)$.\\
Recalling  that
\begin{equation}\label{energie}
\displaystyle{E(u(t)) = E(u_{0}) - a\int_{0}^{t}K(u(\tau))d\tau, \quad t \in [0,T[,}
\end{equation}
with $K(u(t))= \displaystyle{\int |(-\Delta)^{\frac{s+1}{2}}u|^2 - Im \int (-\Delta)^s u |u|^{\frac{4}{d}}\overline{u}}$, we obtain that 
$$\displaystyle{E(u(t))  \lesssim E(u_0) +\Bigl|� \int_0^t (-\Delta)^s( u) |u|^{\frac{4}{d}}\overline{u}dx\Bigr| 
\lesssim E(u_0) + \int_0^t \left\|(-\Delta)^s(u)\right\|_{L^2}\|u\|_{L^{\frac{8}{d}+2}}^{\frac{4}{d}+1}}
$$
This last estimate  together with 
$$ \|u\|_{L^{\frac{8}{d}+2}}^{\frac{4}{d}+1} \lesssim  \|\nabla u\|_{L^2}^{2}\|u_0\|_{L^2}^{\frac{4-d}{d}}
\quad  \mbox{ and }\quad \|(-\Delta)^s(u)\|_{L^2} \leq \|\nabla u\|_{L^2}^{2s}\|u\|_{L^2}^{1-2s} \lesssim \|\nabla u \|^{2s}_{L^2}$$
yield
\begin{equation}\label{cc}
E(u(t)) \lesssim   E(u_0) +\int_0^t   \|\nabla u\|_{L^2}^{2+2s}(\tau) \, d\tau .
\end{equation}
Note that assumption \eqref{law} ensures  that 
\begin{equation}\label{ksurnabla}
0 \le \frac{\displaystyle{\int_{0}^{t}}
\left\|\nabla u(\tau)\right\|_{L^2}^{2+2s}d\tau}{\left\|\nabla u(t)\right\|_{L^2(\mathbb{R}^{d})}^{2}} \lesssim   (T-t)^{-2\beta(1+s)+1+2\alpha}
\rightarrow 0 \; \mbox{� as  }�t\nearrow T ,
\end{equation}

Now, let
 $$\rho(t) = \frac{\left\|\nabla Q\right\|_{L^2(\mathbb{R}^{d})}}{\left\|\nabla u(t)\right\|_{L^2(\mathbb{R}^{d})}} \quad \text{and} \quad v(t,x)=\rho^{\frac{d}{2}}u(t,\rho x)$$
and let $(t_{k})_{k}$ be a sequence of positive times such that $t_{k} \nearrow T$. We set 
 $\rho_{k} = \rho(t_{k})$ and $  v_{k} = v(t_{k},.)$. The family $(v_{k})_{k}$ satisfies 
$$\left\|v_{k}\right\|_{L^2(\mathbb{R}^{d})}= \left\|u(t_k,\cdot)\right\|_{L^2(\mathbb{R}^{d})} < \left\|u_{0}\right\|_{L^2(\mathbb{R}^{d})}\leq \left\|Q\right\|_{L^2(\mathbb{R}^{d})}
\quad \text{and} \quad \left\|\nabla v_{k}\right\|_{L^2(\mathbb{R}^{d})} = \left\|\nabla Q\right\|_{L^2(\mathbb{R}^{d})}.$$
\bigskip
The above estimate on $ \left\|v_{k}\right\|_{L^2} $ and \eqref{cc} lead to
$$
 0< \frac{1}{2}(\int |\nabla v_{k}|^2)\bigg(1 - \big(\frac{\int |v_k|^2}{\int |Q|^2}\big)^{2}\bigg)\leq 
E(v_{k}) = \rho^2_{k} E(u(t)) \lesssim \rho^2_{k} E(u_{0}) + \rho^2_{k}  \int_{0}^{t_{k}}\left\|\nabla u(\tau)\right\|_{L^2}^{2+2s}d\tau�
$$
which, together with \eqref{ksurnabla}, ensures that $\displaystyle{\lim_{k \longrightarrow +\infty}} E(v_{k}) =0$. This forces 
\begin{equation}\label{vk ver Q}
\left\|v_{k}\right\|_{L^{\frac{4}{d}+2}}^{\frac{4}{d}+2 } \rightarrow  \frac{d+2}{d}\left\|\nabla v_k\right\|_{L^2(\mathbb{R}^{d})}^{2}=\frac{d+2}{d}\left\|\nabla Q\right\|_{L^2(\mathbb{R}^{d})}^{2}
\end{equation}
and thus the  family $(v_{k})_{k}$ satisfies the hypotheses of Theorem \ref{limsup} with \\
$$m^{\frac{4}{d}+2} = \frac{d+2}{d}\left\|\nabla Q\right\|_{L^2(\mathbb{R}^{d})}^{2} \quad \text{and} \quad M = 
\left\|\nabla Q\right\|_{L^2(\mathbb{R}^{d})}\, .$$
\bigskip
Hence,  there exists a family $(x_{k})_{k} \subset \mathbb{R}$ and a profile $V \in H^{1}(\mathbb{R})$ with $\left\|V\right\|_{L^2(\mathbb{R}^{d})} \geq \left\|Q\right\|_{L^2(\mathbb{R}^{d})}$, such that,
\begin{equation}\label{convergencefaible}
\displaystyle{\rho^{\frac{d}{2}}_{k}u(t_{k}, \rho_{k}\cdot +  x_{k}) \rightharpoonup V \in H^{1} \quad \text{weakly}.}
\end{equation}
Using (\ref{convergencefaible}), $\forall A \geq 0$
\begin{equation}
 \displaystyle{\liminf_{n\to +\infty}\int_{B(0,A)}\rho_{n}^d|u(t_{n},\rho_{n}x+x_{n})|^{2}dx\geq \int_{B(0,A)}|V|^{2}dx.}\nonumber
 \end{equation}
But, since $\displaystyle \lim_{n\to +\infty}\rho_{n}=0$, \;  $\rho_{n}A < 1$ for $ n $ large enough and thus
  
  $$
  \displaystyle{\liminf_{n\to +\infty}\sup_{y\in\mathbb{R}}\int_{|x-y|\leq 1}|u(t_{n},x)|^{2}dx \geq \liminf_{n\to +\infty}
 \int_{|x-x_n| \leq \rho_{n}A}|u(t_{n},x)|^{2}dx \geq \int_{|x|\leq A}|V|^{2}dx.}
  $$
Since this it is true for all $A > 0$ we obtain that 
$$
 \|u_0\|_{L^2}^2>  \liminf_{n\to +\infty}\sup_{y\in\mathbb{R}}\int_{|x-y|\leq 1}|u(t_{n},x)|^{2}dx \ge \|Q\|_{L^2}^2
$$
which contradicts the assumption $ \|u_0\|_{L^2} \le \|Q\|_{L^2} $ and the desired result is proven.

 \section{Blow up solution.}
In this section, we prove the existence of the explosive solutions in the case $ 0<s<1$. 
\begin{theorem}\label{theorem 3} Let $0<s<1$.
 There exist  a set of initial data $\Omega$ in $H^1$, such that  for any $0 < a < a_0$ with $a_0=a_0(s)$ small enough, the emanating solution $u(t)$  to (\ref{NLSas}) blows up in finite time in the log-log regime.
\end{theorem}
The set of initial data $\Omega$  is the set described  in \cite{MerleRaphael1}, in order to initialize the log-log regime. It is open in $H^1$. Using the continuity  with regard to the initial data and the parameters, we easily obtain the following corollary:
\begin{corollary}\label{colloraireperturbation}
 Let $ 0<s<1 $ and $u_{0} \in H^1$ be an initial data  such that the corresponding solution $u(t)$ of (\ref{NLS}) blows up in the loglog regime. There exist $\beta_{0} > 0$ and $a_{0} > 0 $ such that if $v_{0} = u_{0} + h_{0}$, $\left\|h_{0}\right\|_{H^{1}} \leq \beta_{0}$ and $a \leq a_{0}$, the solution $v(t)$ for (\ref{NLSas}) with the initial data $v_{0}$ blows up in finite time.
\end{corollary}
\noindent 
Now to prove Theorem \ref{theorem 3}, 
we look for a solution of (\ref{NLSas}) such that for $t$ close enough to blowup time, we shall have the following decomposition:
\begin{equation}\label{decomposition}
u(t,x)=\frac{1}{\lambda^{\frac{d}{2}}(t)}(Q_{b(t)} + \epsilon)(t,\frac{x-x(t)}{\lambda(t)})e^{i\gamma(t)},
\end{equation}

for some geometrical parameters $(b(t),\lambda(t), x(t),\gamma(t)) \in (0,\infty)\times(0,\infty)\times\mathbb{R}^{d}\times\mathbb{R}$, here $\lambda(t)\sim \frac{1}{\|\nabla u(t)\|_{L^2}}$,
 and the profiles $Q_{b}$ are suitable deformations of $Q$ related to some extra degeneracy
of the problem.\\

 Note that  we will abbreviated our proof because it is very  close to the case of  
$s=0$( see Darwich\cite{Darwich}). Actually, as noticed in \cite{Planchon1}, we only need to prove that in the log-log regime
the $L^2$ norm does not grow, and the growth of the energy( resp the momentum) is below $\frac{1}{\lambda(t)^2}$ (resp $\frac{1}{\lambda(t)}$) .
In this paper, we will prove that in the
log-log regime, the growths of the energy and the momentum are bounded by:
$$E(u(t))\lesssim \log (\lambda(t))\lambda(t)^{-2s}, ~~ 	P(u(t)) \lesssim \log(\lambda(t))\lambda(t)^{\frac{-2s}{s+1}}.$$
Let us recall that a fonction u :$[0,T]\longmapsto H^1$ follows the log-log regime if the following
 uniform controls on the decomposition (\ref{decomposition}) hold on $[0,T]$:
\begin{itemize}
\item {Control of $b(t)$}
\begin{equation}\label{b petit}
 b(t) > 0 , ~ b(t) < 10 b(0).
\end{equation}

\item Control of $\lambda$:
\begin{equation}\label{control of lambda}
\lambda(t) \leq e^{-{e^\frac{\pi}{100b(t)}}}
\end{equation}
and the monotonicity of $\lambda$:
\begin{equation}\label{monotonicity}
\lambda(t_2) \leq\frac{3}{2} \lambda(t_{1}), \forall~0\leq t_{1} \leq t_{2} \leq T.
\end{equation}
Let $k_{0} \leq k_{+} $ be integers and $T^{+} \in [0,T]$  such that
\begin{equation}\label{lambda 0 et lambda T}
\frac{1}{2^{k_{0}}} \leq \lambda(0) \leq \frac{1}{2^{k_{0}-1}}, \frac{1}{2^{k_{+}}} \leq \lambda(T^{+}) \leq \frac{1}{2^{k_{+}-1}}
\end{equation}
and for $k_{0} \leq k \leq k_{+} $, let $t_{k}$ be a time such that
\begin{equation}
\lambda(t_{k}) = \frac{1}{2^k},
\end{equation}
then we assume the control of the doubling time interval:
\begin{equation}\label{tk}
t_{k+1} - t_{k} \leq k \lambda^2(t_{k}).
\end{equation}
\item control of the excess of mass:
\begin{equation}\label{controlepsilon}
\int_{\mathbb{R}^{d}}\left|\nabla \epsilon(t)\right|^2 + \int_{\mathbb{R}^{d}}\left|\epsilon(t)\right|^2e^{-\left|y\right|} \leq \Gamma_{b(t)}^{\frac{1}{4}}, \text{where}~~ \Gamma_{b} \sim e^{-\frac{\pi}{b}}.
\end{equation}
\end{itemize}
\bigskip
The main point is to establish that \eqref{b petit}-\eqref{controlepsilon} determine a trapping region for the flow.
Actually, after the decomposition (\ref{decomposition}) of $u$, the
log-log regime corresponds to the following asymptotic controls
\begin{equation}\label{stragerie 1}
b_{s}\sim Ce^{-\frac{c}{b}},\, -\frac{\lambda_{s}}{\lambda} \sim b
\end{equation}
and 
\begin{equation}\label{stragerie 2}
\int_{\mathbb{R}^{d}}|\nabla \epsilon|^2\lesssim e^{-\frac{c}{b}},
\end{equation}
where we have introduced the rescaled time $\frac{ds}{dt} = \frac{1}{\lambda^2}$.

In fact, (\ref{stragerie 2}) is partly a consequence of the
preliminary estimate:
\begin{equation}\label{stragerie 3}
\int_{\mathbb{R}^{d}}|\nabla \epsilon|^2 \lesssim e^{-\frac{c}{b}} + \lambda^2(t) E(t).
\end{equation}
One then observes that in the log-log regime, the integration of the laws (\ref{stragerie 1}) yields
\begin{equation}\label{stragerie 4}
\lambda \sim e^{-e^{\frac{c}{b}}} \ll e^{-{\frac{c}{b}}},
\, b(t) \rightarrow 0, \, t\rightarrow T.
\end{equation}
Hence, the term involving the conserved Hamiltonian is asymptotically negligible
with respect to the leading order term $e^{-{\frac{c}{b}}}$
which drives the decay $(\ref{stragerie 3})$ of $b$. This
was a central observation made by Planchon and Raphael in \cite{Planchon1}. In fact, any growth
of the Hamiltonian algebraically below $\frac{1}{\lambda^{2}(t)}$ would be enough. 
In this paper, we will prove that in the log-log regime, the growth of the energy is estimated by:
\begin{equation}\label{decroissance}
\displaystyle{E(u(t)) \lesssim \big(\log \big(\lambda(t)\big)\big)\lambda^{-2s}(t).}
\end{equation}
It then follows from $(\ref{stragerie 3})$ that:
\begin{equation}
\int_{\mathbb{R}^{d}}|\nabla \epsilon|^2  \lesssim e^{-{\frac{c}{b}}}.
\end{equation}
An important feature of this estimate of $H^1$ flavor is that it relies on a flux computation
in $L^2$ . This allows one to
recover the asymptotic laws for the geometrical parameters $(\ref{stragerie 1})$ and to close the
bootstrap estimates of the log-log regime.
\begin{remark}
Actually, one also needs the bound on the momentum  to control the geometrical parameters (see Lemma 7.2 in \cite{Darwich}).
\end{remark}
\subsection{Control of the
 energy and the kinetic momentum}
Let us recall that we say that an ordered  pair $(q,r) $ is admissible whenever   $\frac{2}{q} + \frac{d}{r} = \frac{d}{2}$ and $2 < q \leq \infty $. 
 We define the Strichartz norm of functions $u : [0, T]\times\mathbb{R}^{d} \longmapsto \mathbb{C}$ by:
\begin{equation}\label{us0}
\left\|u\right\|_{S^{0}([0,T]\times\mathbb{R}^{d})} = \sup_{(q,r) admissible}\left\|u\right\|_{L^{q}_{t}L^{r}_{x}([0,T]\times\mathbb{R}^{d})}
\end{equation}
and
\begin{equation}\label{us1}
\left\|u\right\|_{S^{1}([0,T]\times\mathbb{R}^{d})} =  \sup_{(q,r) admissible}\left\|\nabla u\right\|_{L^{q}_{t}L^{r}_{x}([0,T]\times\mathbb{R}^{d})}
\end{equation}
We will sometimes abbreviate $S^i([0,T]\times \mathbb{R}^{2})$ with $S^{i}_{T}$  or $S^{i}[0, T]$, $i= 1,2$. 
Now we will derive an estimate on the energy, to check that it remains small with respect to $\lambda^{-2}$:

\begin{proposition}\label{control energie}
Assuming that (\ref{control of lambda})-(\ref{controlepsilon}) hold, then the energy and kinetic momentum are controlled on $[0, T^+]$ by:\\
\begin{equation}\label{control de lenergie}
\left|E(u(t))\right|\lesssim \big(\log\big(\lambda(t)\big)\big)\lambda^{-2s}(t),
\end{equation}

\begin{equation}\label{control de moment 1}
\left|P(u(t))\right|\lesssim \big(\log\big(\lambda(t)\big)\big)\lambda^{\frac{-2s}{s+1}}(t).
\end{equation}
\end{proposition}
To prove Proposition \ref{control energie}, we will need the two following lemmas.
\begin{lemma}\label{prop3}
Let $u\in C([0,T]; H^1 (\R^d))$ be a solution of (\ref{NLSas}). Then we have the following estimation:
$$
\|\nabla u\|_{L^\infty_T L^2_x} + \|(\Delta) ^{\frac{s+1}{2}} u \|_{L^2_ T L^2_x} \lesssim \| \nabla u_0\|_{L^2_x} + \||u|^{\frac{4}{d}}\nabla u\|_{L^1_T L^2_x}
$$
\end{lemma}
\begin{proof}
Multiply Equation \ref{NLSas} by $\overline{ \Delta u}$, integrate and take the imaginary part, to obtain :\\
$$ \frac{1}{2}\frac{d}{dt}\| \nabla u \|^2 + a \int |(-\Delta)^{\frac{s+1}{2}} u|^2 \leq |\int |u|^{\frac{4}{d}} u \Delta u| =|\int \nabla(|u|^{\frac{4}{d}}u )\nabla u|$$

By integrating in time, we get 
$$
\frac{1}{2}\| \nabla u \|^2_{L^\infty_T L^2} + a \|(-\Delta)^{\frac{s+1}{2}} u\|^2_{L_{T}^2L_x^2} \leq \frac{1}{2}\| \nabla u_0 \|^2_{L^2}  +\|\nabla (|u|^{\frac{4}{d}}u)\|_{L_T^1L_x^2}\|\nabla u\|_{L^{\infty}_T {L^2}}
$$

Dividing by $\sqrt{\frac{1}{2}\| \nabla u \|^2_{L^\infty_T L^2} + a \|(-\Delta)^{\frac{s+1}{2}} u\|^2_{L_{T}^2L_x^2}}$
we obtain:
$$
\|\nabla u\|_{L^\infty_T L^2_x} + \|(\Delta) ^{\frac{s+1}{2}} u \|_{L^2_T L^2_x} \lesssim \| \nabla u_0\|_{L^2} + \||u|^{\frac{4}{d}}\nabla u\|_{L^1_T L^2_x}
$$
\end{proof}
\begin{lemma}\label{us1} There exists a real number $ 0<\alpha \ll 1 $ such that the following holds: 
Let  $u\in C([0,T]; H^1(\R^d))$  be the  solution of $(\ref{NLSas})$ emanating from $u_{0}\in H^1$.  For a fixed $ t\in ]0,T[ $ we set   $\Delta t = \alpha\, \left\|u_{0}\right\|^{\frac{d-4}{d}}_{L^2}\left\|u(t)\right\|_{H^{1}}^{-2}$. Then $ u\in C([t,t+\Delta t]; H^1(\R^d))$ and we have the following controls\\
$$\left\|u\right\|_{{S^{0}[t,t+\Delta t]}} \leq 2\left\|u_{0}\right\|_{L^2} 
$$
and 
$$
 \left\|u\right\|_{{S^{1}[t,t+\Delta t]}} + 
\|(-\Delta)^{\frac{s+1}{2}}u\|_ {L^2(]t,t+\Delta t[)L^2_x} \leq 2 \left\|u(t)\right\|_{H^1(\mathbb{R}^{d})}
$$
\end{lemma}
\begin{proof} We first assume that  $ \Delta t>0 $ is such that $ t+\Delta t< T $.
Then, according to Lemma \ref{LpLp}, it holds 
$$\left\|\int_{0}^{t} S_{a,s}(t-\tau)|u(\tau)|^{\frac{4}{d}}u(\tau)  d\tau \right\|_{S^{1}[t,t+\Delta t]} \lesssim \left\||u|^{\frac{4}{d}}\nabla u\right\|_{L^1(]t,t+\Delta t[)L^2_x}\, .$$
Using the  H\"older inequality we obtain:
\begin{equation}
\displaystyle{\big(\int\left|u\right|^{\frac{8}{d}}\left|\nabla u\right|^2\big)^{\frac{1}{2}} \leq \big(\int \left|u\right|^{\frac{2(4+d)}{d}}\big)^{\frac{2}{4+d}}\big(\int \left|\nabla u\right|^{\frac{2(4+d)}{d}} \big)^{\frac{d}{2(4+d)}}.}\nonumber
\end{equation}
Integrating in time and applying again H\"older inequality we get:
\begin{align}
\left\|\left|u\right|^{\frac{4}{d}}\nabla u\right\|_{{L^{1}]t,t+\Delta t[)}L^{2}(\mathbb{R}^{d})}&\leq \bigg(\int\big(\int \left|u\right|^{\frac{2(4+d)}{d}}dx\big)^{\frac{2}{4+d}\frac{4+d}{4}}dt\bigg)^{\frac{4}{4+d}}\nonumber\\
&\times \bigg(\int\big(\int \left|\nabla u\right|^{\frac{2(4+d)}{d}}dx\big)^{\frac{d}{2(4+d)}\frac{4+d}{d}}dt\bigg)^{\frac{d}{4+d}}.\nonumber
\end{align} 
Thus:
\begin{equation}
\displaystyle{\left\|\left|u\right|^{\frac{4}{d}}\nabla u\right\|_{L^{1}(]t,t+\Delta t[)L^{2}(\mathbb{R}^{d})} \leq \left\|u\right\|^{\frac{4}{d}}_{L^{\frac{4+d}{d}}(]t,t+\Delta t[)L^{\frac{8+2d}{d}}(\mathbb{R}^{d})}\left\|\nabla u\right\|_{L^{\frac{4+d}{d}}(]t,t+\Delta t[)L^{\frac{8+2d}{d}}(\mathbb{R}^{d})}}.\nonumber
\end{equation}
But $(\frac{4+d}{d}, \frac{8+2d}{d})$ is admissible, thus we have:
\begin{equation}
\displaystyle{\left\|\left|u\right|^{\frac{4}{d}}\nabla u\right\|_{{L^{1}([t,t+\Delta t])}L^{2}(\mathbb{R}^{d})} \leq \left\|u\right\|^{\frac{4}{d}}_{L^{\frac{4+d}{d}}([t,t+\Delta t])L^{\frac{8+2d}{d}}(\mathbb{R}^{d})}\left\|u\right\|_{S^{1}[t,t+\Delta t]}}.\nonumber
\end{equation}
By Sobolev inequalities we have:
\begin{align}
\left\|u\right\|_{L^{\frac{4+d}{d}[t,t+\Delta t]}L^{\frac{8+2d}{d}}(\mathbb{R}^{d})} &\lesssim \left\|u\right\|_{L^{\frac{4+d}{d}}([t,t+\Delta t])H^{\frac{2d}{d + 4}}(\mathbb{R}^{d})}\nonumber\\
&\leq (\Delta t)^{\frac{d}{d+4}}\left\|u\right\|_{L^{\infty}([t,t+\Delta t])H^{\frac{2d}{d + 4}}(\mathbb{R}^{d})}.\nonumber
\end{align}
Now by interpolation we obtain for $d=1,2,3,4$:
\begin{equation}
\displaystyle{\left\|u\right\|_{L^{\frac{4+d}{d}}([t,t+\Delta t])L^{\frac{8+2d}{d}}(\mathbb{R}^{d})} \leq (\Delta t)^{\frac{d}{d+4}}\left\|u\right\|_{L^{\infty}([t,t+\Delta t])L^{2}(\mathbb{R}^{d})}^{\frac{4-d}{d+4}}\left\|u\right\|_{L^{\infty}([t,t+\Delta t])H^{1}(\mathbb{R}^{d})}^{\frac{2d}{d+4}}}\, , \nonumber
\end{equation}
 which,  according to (\ref{masse}), leads to 
 \begin{equation}\label{estis1}
\|u^{\frac{4}{d}}\nabla u\|_{L^1_tL^2_x} \leq (\Delta t)^{\frac{d}{d+4}}\left\|u_0\right\|_{L^2(\mathbb{R}^{d})}^{\frac{4-d}{d+4}}\left\|u\right\|_{S^{1}[t, t + \Delta t]}^{\frac{2d}{d+4}}\left\|u\right\|_{S^{1}[t,t+\Delta t]}\, .
\end{equation}
Since by Lemma \ref{prop3} it holds 
$$\|u\|_{L^\infty (]t,t+\Delta t[)H^1}  + \|(\Delta) ^{\frac{s+1}{2}} u \|_{L^2(]t,t+\Delta t[)L^2_x} \lesssim \|  u(t)\|_{H^1} + \||u|^{\frac{4}{d}}\nabla u\|_{L^1(]t,t+\Delta t[)L^2_x} \; , 
$$
we finally get 
\begin{align*}
\left\|u\right\|_{S^{1}[t, t + \Delta t]} + \|(-\Delta)^{\frac{s+1}{2}} u\|_{L^2(]t,t+\Delta t[)L^2} &\lesssim \|u(t)\|_{H^1} +(\Delta t)^{\frac{d}{d+4}}\left\|u_0\right\|_{L^2(\mathbb{R}^{d})}^{\frac{4-d}{d+4}}\left\|u\right\|_{S^{1}[t, t + \Delta t]}^{\frac{2d}{d+4}+1}\, .
\end{align*}
In view of  \eqref{oo} and a continuity argument, it follows  that $ u\in C([t,t+\Delta t]; H^1(\R^d)) $   for some $\Delta t \sim \left\|u_0\right\|_{L^2(\mathbb{R}^{d})}^{\frac{d-4}{d}}\left\|u(t)\right\|^{-2}_{H^1(\mathbb{R}^{d})}$ and 
$$
\left\|u\right\|_{S^{1}[t, t + \Delta t]} + \|(-\Delta)^{\frac{s+1}{2}} u\|_{L^2(]t,t+\Delta t[)L^2} \leq 2\|u_0\|_{H^1} \; .
$$
In the same way
 $$
\left\|u\right\|_{S^{0}[t,t+\Delta t]} \lesssim \left\|u(t)\right\|_{L^2(\mathbb{R}^{d})} + (\Delta t)^{\frac{d}{d+4}}\left\|u_0\right\|^{\frac{4-d}{4+d}}_{L^2(\mathbb{R}^{d})} \left\|u\right\|_{S^{1}[t,t+\Delta t]}^{2\frac{d}{d+4}}\left\|u\right\|_{S^{0}[t,t+\Delta t]}
$$
which ensures that 
$$
\left\|u\right\|_{S^{0}[t,t+\Delta t]} \leq 2 \left\|u_{0}\right\|_{L^2(\mathbb{R}^{d})}\; .
$$
\end{proof}

\noindent
{\it Proof of Proposition \ref{control energie} } : According to \eqref{tk}, each interval $[t_k, t_{k+1}]$, can be
divided into $k$ intervals, $[\tau_k^j,\tau_k^{j+1}]$ of length less that $ \lambda(t_k) $. From (\ref{energie}), we have 

$$|E(u(\tau_k^{j+1})) - E(u(\tau_k^{j}))| \lesssim  \Bigl|\int_{\tau_k^{j}}^{\tau_k^{j+1}}\int_{\R^d}� (-\Delta)^s u\,  |u|^{\frac{4}{d}}\overline{u}dx \,\Bigr| \, .$$
For notation convenience we set $ \Theta=]\tau_k^{j}, \tau_k^{j+1}[\times\R^d$. By  Cauchy-Schwarz, it holds 
 $$\int_{\tau_k^{j}}^{\tau_k^{j+1}}\int_{\R^d}(-\Delta)^s u |u|^{\frac{4}{d}}\overline{u} dx\, dt \leq \|(-\Delta)^{\frac{s}{2}}(u|u|^{\frac{4}{d}})\|_{L^{\frac{4+2d}{d}}(\Theta)}
 \|(-\Delta)^{\frac{s}{2}}u\|_{L^{\frac{4+2d}{d}}(\Theta)}$$
 and, by interpolation, we have 
 $$
\|(-\Delta)^{\frac{s}{2}}u\|^2_{L^{\frac{4+2d}{d}}(\Theta)} \leq \|\nabla u\|^{2s}_{L^{\frac{4+2d}{d}}(\Theta)}\|u\|_{L^{\frac{4+2d}{d}}(\Theta)}^{2-2s}\, .
$$
Noticing that the fractional Leibniz rule leads to 
$$\|(-\Delta)^{\frac{s}{2}}(|u|^{\frac{4}{d}}u)\|_{L^{\frac{4+2d}{4+d}}(\Theta)} \lesssim  \|(-\Delta)^{\frac{s}{2}}u\|_{L^{\frac{4+2d}{d}}(\Theta)}
\|u^{\frac{4}{d}}\|_{L^{\frac{4+2d}{4}}(\Theta)}\lesssim  \|(-\Delta)^{\frac{s}{2}}u\|_{L^{\frac{4+2d}{d}}}\|u\|^{\frac{4}{d}}_{L^{\frac{4+2d}{d}}(\Theta)}\, , $$
 we finally obtain 
$$|E(u(\tau_k^{j+1})) - E(u(\tau_k^{j}))| \lesssim \|u\|_{L^{\frac{4+2d}{d}}(\Theta)}^{\frac{4}{d}}\|\nabla u\|^{2s}_{L^{\frac{4+2d}{d}}(\Theta)}
\|u\|_{L^{\frac{4+2d}{d}}(\Theta)}^{2-2s}\, .$$
Since $(\frac{4}{d}+2,\frac{4}{d}+2)$ is an admissible pair,  Lemma \ref{us1}  yields 
$$|E(u(\tau_k^{j+1})) - E(u(\tau_k^{j}))| \lesssim \lambda(t_k)  ^{-2s}$$
and summing  over $ j $  we get 
$$|E(u(t_{k+1})) - E(u(t_k)) |\lesssim  k \lambda(t_k)  ^{-2s}\, .$$
Finally,  taking $T^+ = T$ and summing from $k_0$ to $k^+$, we obtain:
$$
|E(u(T^+)) - E(u_0) |\lesssim k^+\lambda^{-2s}(T^+)\lesssim  \log (\lambda(T))\lambda^{-2s}(T).
$$
Note that the growth of the energie is small with to respect $\frac{1}{\lambda^2}$, because $s<1$.\\
Let us now proceed with the momentum. According to (\ref{momentum}) we have :

$$|P(u(\tau_k^{j+1})) - P(u(\tau_k^{j}))|\lesssim \int_{\tau_k^{j}}^{\tau_k^{j+1}} \Bigl| \int_{\R^d} \overline{\nabla u} (-\Delta)^s u \Bigr| \, .
$$
But
$$ \Bigl| \int_{\R^d}  \overline{\nabla u}(-\Delta)^s u \Bigr| =  \Bigl|\int_{\R^d} (-\Delta)^{\frac{s}{2}+\frac{1}{4}}u (-\Delta)^{\frac{s}{2}-\frac{1}{4}}\overline{\nabla u }  \Bigr| \leq \|(-\Delta)^{\frac{s}{2}+\frac{1}{4}}u\|_{L^2(\R^d)}\|(-\Delta)^{\frac{s}{2}-\frac{1}{4}}\nabla u\|_{L^2(\R^d)} $$
with 
$$\|(-\Delta)^{\frac{s}{2}-\frac{1}{4}}\nabla u\|_{L^2(\R^d)} = \|(-\Delta)^{\frac{s}{2}+\frac{1}{4}}u\|_{L^2(\R^d)}$$
and, by interpolation,
$$
\|(-\Delta)^{\frac{s}{2}+\frac{1}{4}}u\|^2_{L^2(\R^d)}  \leq \|(-\Delta)^{\frac{s}{2}+\frac{1}{2}}u\|^{2\theta}_{L^2(\R^d)}\|u\|_{L^2(\R^d)}^{2-2\theta}\, , 
$$
where $0 < \theta= \frac{2s+1}{2s+2} < 1$.
Therefore we get 
$$
|P(u(\tau_k^{j+1})) - P(u(\tau_k^{j}))|\lesssim  (\tau_{k+1}-\tau_k)^{1-\theta} \|u_0\|_{L^2(\R^d)}^{2-2\theta}
\|(-\Delta)^{\frac{s}{2}+\frac{1}{2}}u\|^{2\theta}_{L^2(\Theta)}
$$
and Lemma \ref{us1} ensures  that
$$
|P(u(\tau_k^{j+1})) - P(u(\tau_k^{j}))|\lesssim   \lambda^{2-2\theta} \lambda^{-2\theta} = \lambda^{2-4\theta} = \lambda^{\frac{-2s}{s+1}}\, .
$$
Summing over $ j $ we obtain that:
$$
|P(u(\tau_k) - P(u(\tau_k))|\lesssim    k  \lambda^{\frac{-2s}{s+1}}
$$
and summing from $k_0$ to $k^+$, we finally get 
$$
|P(u(T^+)) - P(u_0))|\lesssim \log(\lambda(T))\lambda(T)^{\frac{-2s}{s+1}} \; .
$$
Note that the growth of the momentum  is small with respect $\frac{1}{\lambda}$ since $1-\frac{2s}{s+1} > 0.$




\end{document}